\documentclass[letterpaper,reqno,12pt]{amsart}
\usepackage[margin=1.2in]{geometry}

\usepackage{amssymb, enumerate, color}
\usepackage{mathrsfs}
\usepackage[all]{xy}
\usepackage{hyperref}
\usepackage{enumitem}

\usepackage{amsmath}

\numberwithin{equation}{section} 

\theoremstyle{plain}
\newtheorem{thm}{Theorem}[section] 
\newtheorem{cor}[thm]{Corollary}
\newtheorem{prop}[thm]{Proposition}

\newtheorem{lem}[thm]{Lemma}

\theoremstyle{definition} 
\newtheorem{defn}[thm]{Definition}
\newtheorem{lem-defn}[thm]{Lemma-Definition}
\newtheorem{setting}[thm]{Setting}
\newtheorem{eg}[thm]{Example} 

\newtheorem*{notation}{Notation}

\theoremstyle{remark}
\newtheorem{rem}[thm]{Remark}

\newtheorem*{claim}{Claim}

\newtheorem*{acknowledgement}{Acknowledgments}
\newtheorem*{claimproof}{Proof of Claim}

\def\ge{\geqslant}
\def\le{\leqslant}
\def\phi{\varphi}
\def\epsilon{\varepsilon}
\def\tilde{\widetilde}
\def\bar{\overline}
\def\mapsto{\longmapsto}

\def\mod{\operatorname{\,mod}}

\newcommand{\F}{\mathbb{F}}
\newcommand{\N}{\mathbb{N}}
\newcommand{\Q}{\mathbb{Q}} 
\newcommand{\C}{\mathbb{C}} 
\newcommand{\R}{\mathbb{R}} 
\newcommand{\Z}{\mathbb{Z}}

\newcommand{\ba}{\mathfrak{a}}
\newcommand{\bb}{\mathfrak{b}}
\newcommand{\m}{\mathfrak{m}}
\newcommand{\n}{\mathfrak{n}}
\newcommand{\p}{\mathfrak{p}}
\newcommand{\q}{\mathfrak{q}}

\newsavebox{\circlebox}
\savebox{\circlebox}{\fontencoding{OMS}\selectfont\Large\char13}
\newlength{\circleboxwdht}

\def\Hom{\operatorname{Hom}}
\def\Spec{\operatorname{Spec}}

\def\Div{\operatorname{div}}

\def\id{\operatorname{id}}

\def\RHom{\operatorname{RHom}}

\def\H{\operatorname{H}}
\def\B{\mathcal{B}}
\DeclareMathOperator*{\ulim}{ulim}
\def\Ann{\operatorname{Ann}}

\def\Ext{\operatorname{Ext}}
\def\J{\mathcal{J}}
\def\upf{\operatorname{upf}}
\title{$F$-pure and $F$-injective singularities in equal characteristic zero}

\author{Tatsuki Yamaguchi}
\address{Graduate School of Mathematical Sciences, University of Tokyo, 3-8-1 Komaba, Meguro-ku, Tokyo 153-8914, Japan}
\email{tyama@ms.u-tokyo.ac.jp}


\subjclass[2020]{}

\begin{document}

\begin{abstract}
	Inspired by Schoutens' results, we introduce a variant of sharp $F$-purity and sharp $F$-injectivity in equal characteristic zero via ultraproducts. As an application, we show that if $R\to S$ is pure and $S$ is of dense $F$-pure type, then $R$ is of dense $F$-pure type.
\end{abstract}

\maketitle

\section{Introduction}
A ring homomorphism $R\to S$ is said to be {\it pure} if for any $R$-module $M$, the natural map $M \to M\otimes_R S$ is injective. Examples of pure maps include split maps and faithfully flat maps. Geometrically, when a linearly reductive group $G$ acts on an affine variety $Y:=\Spec S$, then the inclusion $R\hookrightarrow S$ from the ring of invariants $R:=S^G$ is pure, where $X:=\Spec R$ is the quotient of $Y$ by the action of $G$. It is natural to ask what properties descend under pure morphisms $R\to S$. We list some known results below:
\begin{enumerate}
	\item Boutot \cite{Bou87} showed that if $R$ and $S$ are essentially of finite type over a field of characteristic zero and if $S$ has rational singularities, then $R$ has rational singularities.
	\item Schoutens \cite{Sch05} showed that if both $R$ and $S$ are $\Q$-Gorenstein normal local domains essentially of finite type over $\C$ and if $S$ has log terminal singularities, then so does $R$.
	On the other hand, Braun, Greb, Langlois and Moraga \cite{BGLM} showed that if  $S$ is of klt type and a linearly reductive group $G$ acts on $S$, then $S^G$ is also of klt type. Here singularities of klt type are a natural generalization of log terminal singularities to non-$\Q$-Gorenstein setting.
	Recently, Zhuang \cite{Zhu} generalized the above two results: if $S$ is of klt type, then $R$ is of klt type.
	\item Godfrey and Murayama \cite{GM22} showed that if $R$ and $S$ are essentially of finite type over $\C$, and if $S$ has Du Bois singularities, then $R$ has Du Bois singularities. This is a generalization of a result of Kov\'acs \cite{Kov}, who showed the same result when the morphism $R\to S$ splits.
\end{enumerate}
 These raise the question of whether log canonicity descends under pure morphisms (\cite[Question 2.11]{Zhu}). Some partial affirmative answers are known: for example, see \cite{GM22}, \cite{TY}. Log canonical singularities are closely related to $F$-pure singularities, introduced by Hochster and Roberts \cite{HR76}, which are singularities in positive characteristic defined by the purity of the Frobenius morphism. We say that a ring essentially of finite type over a field of characteristic zero is of dense $F$-pure type if its modulo $p>0$ reduction is $F$-pure for ``infinitely many $p$'' (for the precise definition of the singularities of dense $F$-pure type, see Definition \ref{definition of dense $F$-pure type}). Hara and Watanabe \cite{HW} showed that if a $\Q$-Gorenstein normal ring is of dense $F$-pure type, then it has log canonical singularities. Takagi \cite{Tak13} showed that the converse is also true if the weak ordinarity conjecture, proposed by Musta\c{t}\u{a} and Srinivas \cite{MSr11}, holds true.
 The aim of this paper is to show that being of dense $F$-pure type descends under pure morphisms.
  For this purpose, we use ultraproducts, which are a fundamental notion in non-standard analysis. Schoutens \cite{Sch05} used them to show that log terminal singularities descend under pure morphisms if the rings are $\Q$-Gorenstein. The author of this paper \cite{Yam} generalized his result in terms of multiplier ideals: a key observation is that if $R$ is $\Q$-Gorenstein normal local domain essentially of finite type over $\C$, then the BCM test ideal $\tau_{\B(R)}(R)$ associated to $\B(R)$, a big Cohen-Macaulay $R$-algebra constructed by Schoutens \cite{Sch04} using ultraproducts, is equal to the multiplier ideal $\J(\Spec R)$. As an application of this observation, he showed that if $R\to S$ is a pure local $\C$-algebra homomorphism between $\Q$-Gorenstein normal local domains essentially of finite type over $\C$, then we have $\J(\Spec S)\cap R\subseteq \J(\Spec R)$. We use a similar but more refined technique to study singularities of dense $F$-pure type. 
  
  $F$-purity is generalized by Takagi \cite{Tak04} to pairs $(R,\ba^t)$ of rings $R$ of positive characteristic and nonzero ideals $\ba$ of $R$ with real exponents $t>0$. The notion of sharp $F$-purity is a variant of Takagi's $F$-purity introduced by Schwede \cite{Schw08}, which behaves better in a geometric setting. For a given property $P$ defined for rings of characteristic $p>0$, singularities of dense $P$-type is defined in a similar way to the case of singularities of dense $F$-pure type.
  
  Our main theorem is stated as follows:
  \begin{thm}\label{main theorem 1}
 	Let $R\to S$ be a pure local $\C$-algebra homomorphism between reduced local rings essentially of finite type over $\C$, $\ba$ be a nonzero ideal of $R$ and $t$ be a positive real number. Suppose that $R$ is $\Q$-Gorenstein normal and $(S,(\ba S)^t)$ is of dense sharply $F$-pure type. Then $(R,\ba^t)$ is of dense sharply $F$-pure type.
 \end{thm}
 This theorem implies that log canonicity descends under pure morphisms between $\Q$-Gorenstein rings if the weak ordinarity conjecture holds true.
To show this theorem, we introduce the notion of ultra-$F$-purity, a variant of $F$-purity in equal characteristic zero via ultraproducts, defined by the purity of the ultra-Frobenii. We use the ultra-perfect closure $R^{\upf}$ (see Definition \ref{definiton of ultra-Frobenii and ultra-perfect closures}), an analogue of the big Cohen-Macaulay $R$-algebra $\B(R)$, to prove the equivalence of ultra-$F$-purity and being of dense $F$-pure type when the ring is $\Q$-Gorenstein. Since $R^{\upf}$ is not necessarily Cohen-Macaulay, we need to consider $\operatorname{p}$-standard sequences introduced by Kawasaki \cite{Kaw}, instead of regular sequences.

In the latter half of this paper, we consider a similar problem for $F$-injective singularities, which are considered a positive characteristic counterpart of Du Bois singularities. They originate in the study of $F$-purity and rational singularities by Fedder \cite{Fed83}. Schwede \cite{Schw09} showed that if a ring of finite type over a field of characteristic zero has dense $F$-injective type, then it has Du Bois singularities. The converse is wide open, which is shown to be equivalent to the weak ordinarity conjecture in \cite{BST}.

 We introduce the notion of ultra-$F$-injectivity, a variant of $F$-injectivity in equal characteristic zero, defined in a similar way to ultra-$F$-purity. It follows from a similar argument that ultra-$F$-injectivity is equivalent to being of dense $F$-injective type if the residue field is isomorphic to $\C$. This equivalence enables us to show that singularities of dense $F$-injective type descend under strongly pure morphisms introduced in \cite{CGM16}. Here a ring homomorphism $R\to S$ is said to be {\it strongly pure} if for any prime ideal $\q$ of $S$, the induced morphism $R_{\q \cap R} \to S_\q$ is pure. The conclusion is stated as follows:
	
	\begin{thm}
		Let $(R,\m)\to (S,\n)$ be a strongly pure local $\C$-algebra homomorphism between reduced local rings essentially of finite type over $\C$, $\ba$ be a nonzero ideal of $R$ and $t$ be a positive real number. Suppose that $R/\m \cong \C$. If $(S,(\ba S)^t)$ is of dense sharply $F$-injective type, then $(R,\ba^t)$ is of dense sharply $F$-injective type.
	\end{thm}
	Note that strong purity is strictly stronger than purity and $F$-injectivity does not descend under pure morphisms (see \cite[Example 8.6]{MP} and \cite{Wat97}).
	
	This paper is organized as follows: in the preliminary section, we quickly review basic notions concerning $F$-singularities and ultraproducts. In Section 3, we review the definition of $\operatorname{p}$-standard sequences and apply them to our non-standard setting. In section 4, we introduce ultra-$F$-purity and prove the main theorem. In section 5, we discuss ultra-$F$-injectivity and some result about singularities of dense $F$-injective type.
\begin{acknowledgement}
	The author is partially supported by JSPS KAKENHI Grant Number 22J13150. He would like to thank his supervisor Shunsuke Takagi for his continued support and advice. He would also like to express his gratitude to Takesi Kawasaki and Takumi Murayama for helpful discussion. He would also like to thank Linquan Ma, Karl Schwede and Ziquan Zhuang for valuable comments.
\end{acknowledgement}

\begin{notation}
Throughout this paper, all rings are assumed to be commutative and with unit element.
\end{notation}

\section{Preliminaries}

This section provides preliminary results needed for the rest of the paper. 

\subsection{$F$-pure and $F$-injective singularities}
This subsection includes the definitions of notions concerning singularities in positive characteristic.

For a ring $R$ of characteristic $p>0$ and a positive integer $e$, we use $F^e_*R$ to denote $R$ regarded as an $R$-module via the $e$-th iterated Frobenius morphism $F^e:R\to R$, i.e., elements in $F^e_* R$ is denoted by $F^e_* r$ where $r$ is an element of $R$, and the $R$-module structure is defined via $a \cdot F^e_* r=F^r_*(a^{p^e}r)$ for any $a\in R$. Under this notation, $F^e:R\to F^e_*R$ is an $R$-module homomorphism.
\begin{defn}
	Let $R$ be a ring of characteristic $p>0$. $R$ is said to be {\it $F$-finite} if the Frobenius morphism $F:R\to F_*R$ is finite.
\end{defn}
\begin{rem}
	Every $F$-finite Noetherian ring is excellent (see \cite{Kun}) and has a dualizing complex (see \cite[Remark 13.6]{Gab04}, \cite[Theorem 10.9]{MP}).
\end{rem}
\begin{defn}[\cite{HR76}, \cite{Schw08}, \cite{Tak04}]
	Let $R$ be a Noetherian ring of characteristic $p>0$, $\ba$ be an ideal of $R$ such that $\ba\cap R^\circ \neq \emptyset$, where $R^\circ$ denotes the set of elements of $R$ not in any minimal prime of $R$, and $t$ be a positive real number.
	\begin{enumerate}
		\item $R$ is said to be {\it $F$-pure} if the Frobenius morphism $F:R\to F_*R$ is pure.
		\item $(R,\ba^t)$ is said to be {\it sharply $F$-pure} if for infinitely many $e\in N$, there exists $f\in \ba^{\lceil t(p^e-1)\rceil}$ such that $\cdot F^e_* f: R\to F^e_*R$ is pure.
	\end{enumerate}
\end{defn}
\begin{rem}
	Schwede gave a refined definition of sharp $F$-purity in \cite{Schw10}, which is equivalent to the above definition if the ring $R$ is local.
\end{rem}
\begin{defn}
	Let $R$ be an $F$-finite reduced ring of characteristic $p>0$, $\ba\subseteq R$ be an ideal such that $\ba \cap R^\circ \neq \emptyset$, and $t$ be a positive real number. We define $\sigma(R,\ba^t)$ as follows:
	\[
	\sigma(R,\ba^t)=\sum_{e\ge1}\sum_{\phi} \phi(F^e_*\ba^{\lceil t(p^e-1)\rceil}), 
	\]
	where $\phi$ runs through all elements of $\Hom_R(F^e_*R,R)$.
\end{defn}
\begin{rem}
	This definition is different from more complicated one in \cite{FST}. $\sigma$ in {\it loc. cit.} was shown to be contained in a non-lc ideal for sufficiently large $p>0$ after reduction modulo $p>0$.
\end{rem}
\begin{prop}\label{proposition of non-f-pure ideal}
	Let $(R,\m)\to (S,\n)$ be a flat local homomorphism between $F$-finite reduced local rings of characteristic $p>0$ such that the induced morphism $R/\m \to S/\m S$ is a separable field extension. Suppose that $\ba\subseteq R$ is an ideal such that $\ba\cap R^\circ\neq \emptyset$, and $t$ is a positive real number. Then we have the following:
	\begin{enumerate}
		\item $(R,\ba^t)$ is sharply $F$-pure if and only if $\sigma(R,\ba^t)=R$. 
		\item For any $\p\in \Spec R$, we have $\sigma(R_\p,(\ba R_\p)^t)=\sigma(R,\ba^t)R_{\p}$.
		\item $\sigma(\widehat{R},(\ba \widehat{R})^t)=\sigma(R,\ba^t)\widehat{R}$.
		\item $\sigma(S,(\ba S)^t)=\sigma(R,\ba^t)S$.
	\end{enumerate}
\end{prop}
\begin{proof}
	The conclusion follows from an argument similar to \cite[Proposition 14.10]{FST} and \cite[Lemma 1.5]{ST15}.
\end{proof}
\begin{defn}[{\cite{Fed83}, \cite[Definition 2.8]{STV}}]\label{definition F-inj}
	Let $(R,\m)$ be a Noetherian local ring of characteristic $p>0$, $\ba$ be an ideal of $R$ such that $\ba\cap R^\circ\neq \emptyset$ and $t$ be a positive real number.
	\begin{enumerate}
		\item $R$ is said to be {\it $F$-injective} if for any $i\in \Z$, $F:\H_\m^i(R)\to \H_\m^i(F_*R)$ is injective.
		\item $(R,\ba^t)$ is said to be {\it sharply $F$-injective} if for any $i\in \Z$ and a nonzero element $\eta\in \H_\m^i(R)$, for infinitely many $e\in \N$, there exists $f\in \ba^{\lceil t(p^e-1) \rceil}$ such that the image of $\eta$ under $\cdot F^e_*f:\H_\m^i(R)\to \H_\m^i(F^e_*R)$ is nonzero.		
	\end{enumerate}
\end{defn}
\begin{defn}\label{definition of F-injective ideal}
	With notation as in Definition \ref{definition F-inj}, suppose that $R$ is $F$-finite and $\omega_R^{\bullet}$ is the normalized dualizing complex of $R$. For $i\in \Z$, $\sigma_{\text{{$F$}-inj}}^{(i)}(\omega^{\bullet}_R,\ba^t)\subseteq h^{-i}\omega_R^{\bullet}$ is defined to be
	\[
	\sum_{e\ge 1} \sum_{f\in \ba^{\lceil t(p^e-1) \rceil}}\operatorname{Im} \left(h^{-i}\RHom_R(F^e_*R,\omega_R^{\bullet})\to h^{-i}\RHom_R(R,\omega_R^{\bullet})\right),
	\]
	where the above morphisms are induced by $\cdot F^e_*f:R\to F^e_*R$.
\end{defn}
\begin{prop} \label{proposition of non-F-injective ideal}
	Let $(R,\m)\to (S,\n)$ be a flat local homomorphism between $F$-finite reduced local rings of characteristic $p>0$ such that the induced morphism $R/\m \to S/\m S$ is a separable field extension. Suppose that $\ba\subseteq R$ is an ideal such that $\ba\cap R^\circ\neq \emptyset$, and $t$ is a positive real number. Then we have the following:
	\begin{enumerate}
		\item $(R,\ba^t)$ is sharply $F$-injective if and only if for all $i$, $\sigma_{\text{{$F$}-inj}}^{(i)}(\omega^{\bullet}_R,\ba^t)= h^{-i}\omega_R^{\bullet}$.
		\item For any $i\in \Z$ and $\p\in \Spec R$, we have $\sigma_{\text{{$F$}-inj}}^{(i)}(\omega^{\bullet}_{R_\p},(\ba R_\p)^t)=\sigma_{\text{{$F$}-inj}}^{(i)}(\omega^{\bullet}_R,\ba^t)_{\p}$.
		\item For any $i\in \Z$, $\sigma_{\text{{$F$}-inj}}^{(i)}(\omega^{\bullet}_{\widehat{R}},(\ba\widehat{R})^t)=\sigma_{\text{{$F$}-inj}}^{(i)}(\omega^{\bullet}_R,\ba^t)\otimes_R \widehat{R}$.
		\item For any $i\in \Z$, $\sigma_{\text{{$F$}-inj}}^{(i)}(\omega^{\bullet}_S,(\ba S)^t)=\sigma_{\text{{$F$}-inj}}^{(i)}(\omega^{\bullet}_R,\ba^t)\otimes_R S\subseteq h^{-i}\omega_S^\bullet$.
	\end{enumerate}
\end{prop}
\begin{proof}
	The conclusion follows from an argument similar to Proposition \ref{proposition of non-f-pure ideal}. Note that since the morphism $R\to S$ is local flat and $\m S=\n$, $\omega_S^\bullet=\omega_R^\bullet\otimes_R S$.
	We also refer the reader to \cite[Proposition 4.3]{Schw09} for details.
\end{proof}
 We explain the definition of models and reductions modulo $p>0$.
\begin{defn}
	Let $R$ be a ring of finite type over $\C$, $\ba$ be an ideal of $R$ and $\p$ be a prime ideal of $R$.
	\begin{enumerate}
		\item A quadruple $(A,R_A,\ba_A,\p_A)$ is said to be a {\it model} of the triple $(R,\ba,\p)$ if the following conditions hold:
		\begin{enumerate}
			\item $A$ is a finitely generated $\Z$-subalgebra of $\C$.
			\item $R_A$ is a finitely generated $A$-algebra such that $R_A\otimes_A\C\cong R$.
			\item $\ba_A$ and $\p_A$ are ideals of $R_A$ such that $\ba=\ba_A R$ and $\p=\p_A R$.
		\end{enumerate}
		\item Let $(A,R_A,\ba_A,\p_A)$ be a model of the triple $(R,\ba,\p)$. For a maximal ideal $\mu$ of $A$, a quadruple $(\kappa(\mu), R_\mu, \ba_\mu,\p_\mu)$ is said to be a {\it reduction modulo $p>0$} if the following conditions hold:
		\begin{enumerate}
			\item $\kappa(\mu)=A/\mu$.
			\item $R_\mu=R_A\otimes_A \kappa(\mu)$.
			\item $\ba_\mu=\ba_A R_\mu$, $\p_\mu=\p_A R_\mu$.
		\end{enumerate}
	\end{enumerate}
\end{defn}
\begin{defn} \label{definition of dense $F$-pure type}
	Let $R$ be a ring of finite type over $\C$, $\ba$ be an ideal such that $\ba \cap R^\circ\neq \emptyset$, $\p$ be a prime ideal of $R$ and $t$ be a positive real number. A pair $(R_\p, (\ba R_\p)^t)$ is said to be of {\it dense sharply $F$-pure (resp. dense sharply $F$-injective) type} if there exists a subset $D$ of $\operatorname{Spm} A$, the set of all maximal ideals of $A$, such that $D$ is a dense subset of $\Spec A$ and, for any $\mu\in D$, $\p_\mu \in \Spec R_\mu$ and the pair $((R_\mu)_{\p_\mu},(\ba_\mu (R_\mu)_{\p_\mu})^t)$ is sharply $F$-pure (resp. sharply $F$-injective).
	When $\ba=R$, we simply say that $R_\p$ is of {\it dense $F$-pure (resp. dense $F$-injective) type} if $(R_\p, R_\p^t)$ is of sharply $F$-pure (resp. sharply $F$-injective) type.
\end{defn}
\begin{rem}
	\begin{enumerate}
		\item This definition depends only on $R_\p$, $\ba R_\p$ and $t$ and is independent of the choice of models. We refer the reader to \cite[Remark 2.5]{MSr11} for details.
		\item Hara and Watanabe \cite{HW} showed that singularities of dense $F$-pure type are log canonical if the ring is $\Q$-Gorenstein. Schwede \cite{Schw09} showed that if a ring of finite type over a field of characteristic zero has dense $F$-injective type, then it has Du Bois singularities. In \cite{MSS}, the result is generalized to the case of rings essentially of finite type over a field of characteristic zero.
	\end{enumerate}
\end{rem}
\subsection{Ultraproducts}
In this subsection, we quickly review the theory of ultraproducts in commutative algebra. We refer the reader to \cite{Sch03} and \cite[Section 3]{Yam} for details.
In this paper, $\mathcal{P}$ denotes the set of prime numbers.
\begin{defn}
	A non-empty subset $\mathcal{F}$ of the power set of $\mathcal{P}$ is said to be a {\it non-principal ultrafilter} if the following four conditions hold.
	\begin{enumerate}
		\item If $A, B \in \mathcal{F}$, then $A\cap B\in \mathcal{F}$.
		\item If $A\in \mathcal{F}$ and $A\subseteq B \subseteq \mathcal{P}$, then $B\in \mathcal{F}$.
		\item For any $A\subseteq \mathcal{P}$, we have $A\in \mathcal{F}$ or $\mathcal{P}\setminus A \in \mathcal{F}$.
		\item For any finite subset $A\subseteq \mathcal{P}$, we have $A\notin \mathcal{F}$.
	\end{enumerate}
\end{defn}
\begin{prop}\label{existence of ultrafilter}
	For any infinite subset $A$ of $\mathcal{P}$, there exists a non-principal ultrafilter $\mathcal{F}$ on $\mathcal{P}$ such that $A\in \mathcal{F}$.
\end{prop}
\begin{proof}
	The conclusion follows from Zorn's lemma.
\end{proof}
From now on, we fix a non-principal ultrafilter $\mathcal{F}$ on $\mathcal{P}$.
\begin{defn}
 Let $\phi$ be a property on $\mathcal{P}$. We say $\phi(p)$ holds {\it for almost all $p$} if the set $\{p\in\mathcal{P}|\text{$\phi(p)$ holds}\}$ belongs to $\mathcal{F}$.
\end{defn}
\begin{defn}
	Let $(A_p)$ be a family of non-empty sets indexed by $\mathcal{P}$. We define the {\it ultraproduct} of $(A_p)$ by
	\[
		A_\infty =\ulim_p A_p :=\prod_{p}A_p/\sim,
	\]
	where $(a_p)\sim (b_p)$ if and only if  $a_p=b_p$ for almost all $p$.
	For $(a_p)\in \prod_p A_p$, we use $\ulim_p a_p$ to denote the equivalent class of $(a_p)$.
\end{defn}
\begin{eg}
	\begin{enumerate}
		\item If $(A_p)$ is a family of rings (resp. fields), then $A_\infty$ is a ring (resp. field).
		\item If $(A_p)$ is a family of rings and $(M_p)$ is a family of $A_p$-modules, then $M_\infty$ is an $A_\infty$-module.
		\item Let $A_p=\N$ for all $p$. We use ${^*\N}$ to denote $\ulim_p A_p$. Then ${^*\N}$ is an ordered semiring and a non-standard model of Peano arithmetic. We regard $\N$ as a subsemiring by the diagonal embedding.
		\item Similarly, we can define an ordered field ${^*\R}$ containing ${^*\N}$. ${^* \R}$ is called the system of hyperreals, which is a fundamental notion in non-standard analysis. We use $\pi$ to denote $\ulim_p p\in {^*\N}$ and, for any $\epsilon=\ulim e_p \in {^* \R}$, we define $\lceil \epsilon \rceil=\ulim_p \lceil e_p \rceil \in {^*\N}$.
	\end{enumerate}
\end{eg}
 Let $\bar{\F_p}$ be an algebraic closure of $\F_p$. There exists a non-canonical isomorphism $\ulim_p \bar{\F_p}\cong \C$. We fix this isomorphism, and let $R$ be a local ring essentially of finite type over $\C$. Then we can construct an approximation of $R$ and the non-standard hull of $R$ (see \cite[Subsection 4.3]{Sch03} and \cite[Definition 3.35]{Yam}). We omit the formal definitions here and provide instead a brief overview on their fundamental properties.
 \begin{enumerate}
 	\item An approximation of $R$ is a family $(R_p)$ of local rings essentially of finite type over $\bar{\F_p}$.
 	\item The non-standard hull $R_\infty$ is the ultraproduct of $(R_p)$. We remark that $R_\infty$ is not Noetherian in general. 
 	\item We have a natural faithfully flat inclusion $R\hookrightarrow R_\infty$.
 	\item An approximation of an element $x$ of $R$ (or $R_\infty$) is a sequence $(x_p)$ of elements $x_p\in R_p$ such that $\ulim_p x_p=x$.
 \end{enumerate}
Let $S$ be a finitely generated $R$-algebra. Then we can construct an $R$-approximation of $S$ and the (relative) $R$-hull of $S$. We refer the reader to \cite[Section 2]{Sch08} and \cite[Subsection 3.3]{Yam} for details. Here, we explain some basic properties of the $R$-hull.
\begin{enumerate}
	\item An $R$-approximation of $S$ is a family $(S_p)$ of finitely generated $R_p$-algebras for almost all $p$.
	\item The relative $R$-hull $S_\infty$ of $S$ is the ultraproduct of $(S_p)$.
	\item We have a natural faithfully flat inclusion $S\hookrightarrow S_\infty$.
\end{enumerate}
For any ideal $\ba$ of $R$, we define an approximation of $\ba$ by using an approximation of a system of generators of $\ba$ as in \cite[Section 4]{Sch03} and \cite[Definition 3.39]{Yam}.
For an $R$-module $M$ and an $S$-module $N$, we define an approximation of $M$ and an approximation of $N$ as in \cite[Subsection 2.5]{Sch05}, \cite[Definition 3.41]{Yam} and \cite[Definition 3.60]{Yam}. 
For any $\R$-divisor $\Delta$ on $\Spec R$, we define an approximation of $\Delta$ as in \cite[Definition 3.45]{Yam}.
We also define an approximation of varieties over $\C$. The reader is referred to \cite{Sch05} and \cite[Subsection 3.3]{Yam} for more details.
Here we collect other useful definitions and introduce a new notion, the ultra-perfect closure.
\begin{defn}\label{definiton of ultra-Frobenii and ultra-perfect closures}
	Let $R$ be a reduced local ring essentially of finite type over $\C$ and $\epsilon =\ulim_p e_p$ be a non-standard integer (i.e. an element of ${^*\N}$).
	\begin{enumerate}
		\item  $F^\epsilon: R\to R_\infty$ is the morphism $R\to R_\infty; x\mapsto \ulim_p x_p^{p^{e_p}}$. We use $F^\epsilon_* R_\infty$ to denote the $R$-module such that $F^\epsilon_* R_\infty$ is isomorphic to $R_\infty$ as an abelian group but, for any $a\in R$ and $b\in R_\infty$, the scalar multiplication on $F^\epsilon_*R_\infty$ is defined by $a\cdot F^\epsilon_* b=F^\epsilon_*(F^\epsilon (a) b)$.
		\item  The {\it ultra-perfect closure} of $R^{\operatorname{upf}}$ is defined to be $\ulim_p R_p^{1/p^\infty}$.
	\end{enumerate}
\end{defn}
\begin{prop}\label{proposition ultra-perfect closure limit}
	Let $R$ be a reduced local ring essentially of finite type over $\C$. Then we have
	\[
	R^{\upf}\cong \varinjlim_{\epsilon\in {^*\N}} F^{\epsilon}_{*} R_\infty.
	\].
\end{prop}
\begin{proof}
	Take $\mu \le \nu \in {^*\N}$ and let $\nu=\ulim_p n_p$, $\mu=\ulim_p m_p$. Since $m_p \le n_p$ for almost all $p$, we have $F^{m_p}R_p\hookrightarrow F^{n_p}R_p \hookrightarrow R_p^{1/p^{\infty}}$ for almost all $p$. Hence, we have
	\[
	F^\mu_*R_\infty \hookrightarrow F^\nu_*R_\infty \hookrightarrow R^{\upf}.
	\]
	Therefore, we can define $\varinjlim_{\epsilon\in {^*\N}} F^{\epsilon}_{*} R_\infty$ and we have $\varinjlim_{\epsilon\in {^*\N}} F^{\epsilon}_{*} R_\infty\hookrightarrow R^{\upf}$. In order to prove the surjectivity, take any $x=\ulim x_p\in R^{\upf}$. For any $p$, there exists $e_p\in \N$ such that $x_p\in F^{e_p}_*R_p$. Let $\epsilon=\ulim e_p\in {^*\N}$. Then we have $x\in \ulim_p F^{e_p}_* R_p\cong F^{\epsilon}_*R_\infty$.
\end{proof}
\begin{defn}[{\cite[Notation 5.1]{Yam22a}}]
	With notation as above, for any $\epsilon=\ulim_p e_p\in {^*\N}$ and an ideal $\ba$ of $R$, $\ba^{\epsilon}$ is defined to be
	\[
		\ulim_p \ba_p^{e_p}.
	\]
\end{defn}
\begin{rem}
	For any $n\in \N$, if $\nu$ is the image of $n$ under the diagonal embedding $\N \hookrightarrow {^*\N}$, then we have $\ba^\nu=\ba^n R_\infty$.
\end{rem}
Following Schoutens \cite[Section 5]{Sch08}, we can define local ultracohomologies. Let $R$ be a local ring essentially of finite type over $\C$ of dimension $d$ and $x_1,\dots,x_d$ be a system of parameters for $R$. Suppose that $M_p$ is an $R_p$-module for almost all $p$ and $M_\infty=\ulim_p M_p$. For $n\in \N$, $1\le i_1<\dots<i_n\le d$, there exists a natural morphism
\[
	(M_\infty)_{x_{i_1}\cdots x_{i_d}}\to \ulim_p (M_p)_{x_{i_1,p}\dots x_{i_d,p}}.
\]
Considering the \v{C}ech complexes, we have commutative diagrams
\[
	\xymatrix{
		\bigoplus_{1\le i_1<\cdots<i_n \le d}(M_\infty)_{x_{i_1}\dots x_{i_n}} \ar[r] \ar[d]& \bigoplus_{1\le j_1<\cdots<j_{n+1}\le d} (M_\infty)_{x_{j_1}\dots x_{j_{n+1}}} \ar[d]\\
		\bigoplus_{1\le i_1<\cdots<i_n \le d}\ulim_p (M_p)_{x_{i_1,p}\dots x_{i_n,p}} \ar[r] & \bigoplus_{1\le j_1<\cdots<j_{n+1}\le d}\ulim_p (M_p)_{x_{j_1,p}\dots x_{j_{n+1},p}}
	}.
\]
Hence, we have a natural morphism
\[
	\H_{\m}^i(M_\infty)\to \ulim_p \H_{\m_p}^i(M_p).
\]
For an element $\eta$ of $\H_{\m}^i(M_\infty)$, a family $(\eta_p)$ of elements of $H_{\m_p}^i(M_p)$ is said to be an {\it approximation} of $\eta$ if $\ulim_p \eta_p$ is equal to the image of $\eta$ under the above natural morphism.
In the next section, we show the injectivity of this map in some situations, which plays an important role in later sections.
\section{$\operatorname{p}$-standard sequences and ultraproducts}
	In this section, we define $\operatorname{p}$-standard sequences following \cite{Kaw} and apply them to the non-standard setting.
\begin{defn}[{\cite[Definition 2.2]{Kaw}}]
	Let $R$ be a Noetherian ring, $M$ be an $R$-module and $d$ be a positive integer. A sequence $x_1,\dots,x_d$ in $R$ is said to be a {\it $\operatorname{p}$-standard sequence} on $M$ if 
	\[
		(x_\lambda^{n_\lambda}|\lambda\in \Lambda)M:x_{i}^{n_i}x_{j}^{n_j}=(x_\lambda^{n_\lambda}|\lambda\in \Lambda)M:x_j^{n_j}
	\]
	for any positive integers $n_1,\dots,n_d$, any subset $\Lambda\subsetneq \{1,...,d\}$ and $i,j\in \{1,\dots,d\}\setminus \Lambda$ such that $i\le j$.
\end{defn}
 Given a Noetherian local ring $(R,\m)$, for a finitely generated $R$-module $M$, the ideal $\ba(M)$ is defined to be
 \[
 	\ba(M)=\prod_{0\le i<\dim M}\Ann_R \H_\m^i(R).
 \]
\begin{defn}[{\cite[Definition 3.1]{Kaw}}]
	Let $R$ be a Noetherian local ring with a dualizing complex, $M$ be a finitely generated $R$-module and $d=\dim M$. A system of parameters $x_1,\dots,x_d$ for $M$ is said to be a {\it $\operatorname{p}$-standard system of parameters} for $M$ if
		\[
			x_i\in \ba(M/(x_{i+1},\dots,x_d)M)
		\]
		for $1\le i\le d$.
\end{defn}
The following are important properties of $\operatorname{p}$-standard systems of parameters.
\begin{prop}\label{proposition of p-standard s.o.p. -> p-standard seq}
	Suppose that $R$ is a Noetherian local ring with a dualizing complex and $M$ is a finitely generated $R$-module. 
	\begin{enumerate}
		\item ({\cite[p. 482]{Cuo}}) There exists a $\operatorname{p}$-standard system of parameters for $M$.
		\item ({\cite[Theorem 3.3]{Kaw}}) A $\operatorname{p}$-standard system of  parameters for $M$ is a $\operatorname{p}$-standard sequence on $M$.
	\end{enumerate}
\end{prop}
\begin{lem}\label{lemma of approximations of Ext}
	Let $R$ be a local ring essentially of finite type over $\C$, and $M$ and $N$ be finitely generated $R$-modules. Then $(\Ext_{R_p}(M_p,N_p))$ is an approximation of $\Ext_R(M,N)$.
\end{lem}
\begin{proof}
	Comparing approximations with reductions modulo $p$, this follows from \cite[Theorem 2.3.5 (e)]{HH}.
\end{proof}
\begin{prop}\label{proposition approximations of p-standard s.o.p.}
	Let $(R,\m)$ be a local ring essentially of finite type over $\C$ and $M$ be a finitely generated $R$-module of $\dim s$. If $x_1,\dots,x_s$ is a $\operatorname{p}$-standard system of parameters for $M$, then $x_{1,p},\dots, x_{s,p}$ is a $\operatorname{p}$-standard system of parameters for $M_p$ for almost all $p$.
\end{prop}
\begin{proof}
	Since $x_{1},\dots,x_s$ is a system of parameters for $M$, $x_{1,p}\dots, x_{s,p}$ is a system of parameters for $M_p$ for almost all $p$. Let $(S,\n)$ be a regular local ring essentially of finite type over $\C$ such that $R$ is isomorphic to a homomorphic image of $S$ and $t=\dim S$. By the local duality, we have
	\[
		\H_\m^i(M)\cong\H_{\n}^i(M)\cong \Hom_S(\Ext_{S}^{t-i}(M,S),E_S)
	\]
	for $1\le i \le t$, where $E_S$ is the injective hull of $S/\n$ as an $S$-module. Hence, we have
	\[
		\Ann_S \H_\m^i(M) =\Ann_S \Ext_S^{t-i}(M,S).
	\]
	Similarly, we have $\Ann_{S_p}\H_{\m_p}^i(M_p)=\Ann_{S_p}\Ext_{S_p}^{t-i}(M_p,S_p)$ for almost all $p$.	By Lemma \ref{lemma of approximations of Ext}, $(\Ext_{S_p}^{t-i}(M_p,S_p))$ is an approximation of $\Ext_S^{t-i}(M,S)$. Given an element $x$ of $\Ann_S\Ext_{S}^{t-i}(M,S)$, we have $x_p\in \Ann_{S_p}\Ext_{S_p}^{t-i}(M_p,S_p)$ for almost all $p$. Therefore, $x_{i,p}\in \ba(M_p/(x_{i+1,p},\dots,x_{s,p})M_p)$ for any $1\le i \le s$ for almost all $p$, which completes the proof.
\end{proof}
\begin{prop}
	Let $R$ be a reduced local ring essentially of finite type over $\C$ and $\epsilon \in {^*\N}$. A $\operatorname{p}$-standard system of parameters $x_1,\dots,x_d$ for $R$ is a $\operatorname{p}$-standard sequence on $F^\epsilon_{*} R_\infty$ and $R^{\upf}$.
\end{prop}
\begin{proof}
	Take any $\epsilon=\ulim_p e_p\in {^*\N}$. For any $n_1,\dots, n_d\in \N$, any subset $\Lambda\subsetneq \{1,\dots,d\}$ and any $i,j\in \{1,\dots,d\}\setminus{\Lambda}$ such that $i\le j$, take $y\in (x_\lambda^{n_\lambda}|\lambda\in \Lambda)F^\epsilon_*R_\infty: x_i^{n_i}x_j^{n_j}$. Suppose that $y=\ulim_p y_p=\ulim_p F^{e_p}_*z_p$. Then we have $x_{i,p}^{n_ip^{e_p}}x_{j,p}^{n_jp^{e_p}}z_p\in (x_{\lambda,p}^{n_\lambda p^{e_p}}|\lambda\in \Lambda)R_p$ for almost all $p$. Since $x_{1,p},\dots,x_{d,p}$ is a $\operatorname{p}$-standard sequence on $R_p$ for almost all $p$ by Proposition \ref{proposition of p-standard s.o.p. -> p-standard seq} and Proposition \ref{proposition approximations of p-standard s.o.p.}, $x_{j,p}^{n_jp^{e_p}}z_p \in (x_{\lambda,p}^{n_\lambda p^{e_p}}|\lambda\in \Lambda)R_p$ for almost all $p$. Therefore, $y_p=F^{e_p}_*z_p \in (x_{\lambda,p}^{n_\lambda}|\lambda\in \Lambda)F^{e_p}_*R_p:x_{j,p}^{n_j}$ for almost all $p$. Hence, $y\in (x_\lambda^{n_\lambda}|\lambda\in \Lambda)F^{\epsilon}_* R_\infty:x_j^{n_j}$. Similarly, we can also show the result for $R^{\upf}$.
\end{proof}
\begin{prop}\label{propositon of modules with p-standard seq}
	Let $(R,\m)$ be a Noetherian local ring of dimension $d$ and $M$ be an $R$-module. Suppose that there exists a system of parameters $x_1,\dots,x_d$ for $R$ such that $x_1,\dots,x_d$ is a $\operatorname{p}$-standard sequence on $M$. Then we have
	\[
		\H_{(x_1,\dots,x_t)}^{i}(M)=\H_\m^0(\H_{(x_1,\dots,x_t)}^i(M))
	\]
	for any $1\le t \le d$ and $i<t$.
\end{prop}
\begin{proof}
	We work by induction on $t$. If $t=1$ and $y\in \H_{(x_1)}^0(M)$, then there exists $n\in \N$ such that $x_1^n y =0$. For any $j\ge 1$, we have $x_1^nx_jy=0$. Since $x_1,\dots, x_d$ is a $\operatorname{p}$-standard sequence on $M$, we have $y\in 0:_M x_1^nx_j=0:_Mx_j$. Hence, $x_j y=0$. Since $x_1,\dots,x_d$ is a system of parameters for $R$, we have $y\in \H_\m^0(\H_{(x_1)}^0(M))$. Next, assume that $t>1$. Take any $n\in \N$ and consider an exact sequence
	\[
		0\to 0:_M x_1^n \to M \xrightarrow{\cdot x_1^n} M \to M/x_1^n\to 0.
	\]
	Since $x_1,\dots,x_d$ is a $\operatorname{p}$-standard sequence on $M$, $(x_1,\dots,x_d)(0:_M x_1^n)=0$. Hence, we get a long exact sequence
	\[
		\xymatrix{
		 0 \ar[r] & 0:_M x_1^n \ar[r] & \H_{(x_1,\dots,x_t)}^0(M) \ar[r]^{\cdot x_1^n}& \H_{(x_1,\dots,x_t)}^0(M)\ar[r] & \H_{(x_1,\dots,x_t)}^0(M/x_1^nM)\\
		 &\ar[r] & \H_{(x_1,\dots,x_t)}^1(M) \ar[r]^{\cdot x_1^n}& \H_{(x_1,\dots,x_t)}^1(M)\ar[r] & \H_{(x_1,\dots,x_t)}^1(M/x_1^nM) \\
		 &\ar[r] &\dots \qquad \qquad& &
		}
	\]
	For any $ i < t$ and for any $\eta\in \H_{(x_1,\dots,x_t)}^i(M)$, there exists $n\in \N$ such that $x_1^n \eta=0$. If $i=0$, then we have $\eta\in \H_\m^0(\H_{(x_1,\dots,x_t)}^0(M))$ since $(x_1,\dots,x_d)(0:_M x_1^n)=0$. If $i>0$, then there exists an element $\xi \in \H_{(x_1,\dots,x_t)}^{i-1}(M/x_1^nM)$ mapped to $\eta$. By the induction hypothesis, there exists $m\in \N$ such that $\m^m\xi=0$. Hence, we have $\m^m\eta=0$.
\end{proof}
The next corollary is a variant of the Nagel-Schenzel isomorphism.
\begin{cor}\label{corollary of local cohomology w.r.t. p-standard s.o.p.}
	With notation as in Proposition \ref{propositon of modules with p-standard seq}, for any $0\le t\le d$, we have
	\[
		\H_{\m}^t(M)=\H_\m^0(\H_{(x_1,\dots,x_t)}^t(M)).
	\]
\end{cor}
\begin{proof}
	Considering the spectral sequence
	\[
		E_2^{ij}=\H_\m^{i}(\H_{(x_1,\dots,x_t)}^j(M)) \Rightarrow E^{i+j}=\H_\m^{i+j}(M),
	\]
	the conclusion follows from Proposition \ref{propositon of modules with p-standard seq} and the proof of \cite[Lemma 3.4]{NS}.
\end{proof}
\begin{prop}\label{proposition local cohomology ultraproducts injective}
	Let $(R,\m)$ be a reduced local ring essentially of finite type over $\C$  and $\epsilon=\ulim_p e_p\in {^*\N}$. Then the morphisms
	\[
		\H_\m^i(F^\epsilon_* R_\infty)\to \ulim_p \H_{\m_p}^i(F^{e_p}_*R_p)
	\]
	and
	\[
		\H_{\m}^i(R^{\upf})\to \H_{\m_p}^i(R_p^{1/p^{\infty}})
	\]
	are injective for any $i\ge 0$.
\end{prop}
\begin{proof}
	Since the proofs are similar, we will only show that
	$\H_\m^i(F^\epsilon_* R_\infty)\to \ulim_p \H_{\m_p}^i(F^{e_p}_*R_p)$ is injective. We may assume $0\le i \le d$. Let $x_1,\dots,x_d$ be a $\operatorname{p}$-standard system of parameters for $R$. By Corollary \ref{corollary of local cohomology w.r.t. p-standard s.o.p.}, we have 
	\[
		\H_\m^i(F^\epsilon_*R_\infty)\cong\H_\m^0(\H_{(x_1,\dots,x_i)}^i(F^{\epsilon}_*R_\infty))
	\]
	and
	\[
		\H_{\m_p}^i(F^{e_p}_*R_p) \cong \H_{\m_p}^0(\H_{(x_1,\dots,x_i)}^i(F^{e_p}_*R_p))
	\]
	for almost all $p$. Considering the \v{C}ech complex, any element $\eta$ of $\H_{(x_1,\dots,x_i)}^i(F^\epsilon_*R_\infty)$  can be represented by
	\[
		\left[\frac{y}{(x_1\cdots x_i)^t}\right],
	\]
	where $y\in F^\epsilon_*R_\infty$ and $t\in \N$.  We show that \[
	\H_{(x_1,\dots,x_i)}^i(F^\epsilon_*R_\infty)\to \ulim_p \H_{(x_{1,p},\dots, x_{i,p})}^i(F^{e_p}_*R_p)
	\]
	is injective. Suppose that the image of $\eta$ in $\ulim_p \H_{(x_{1,p},\dots,x_{i,p})}^i(F^{e_p}_*R_p)$ equals zero. Let $x':=x_1\cdots x_i$. Then there exists $s_p\in \N$ such that $(x'_p)^{s_p}y_p\in (x_{1,p}^{s_p+t},\dots,x_{i,p}^{s_p+t})F^{e_p}_*R_p$. Hence, $(x'_p)^{s_pp^{e_p}}y_p^{p^{e_p}}\in (x_{1,p}^{(s_p+t)p^{e_p}},\dots,x_{i,p}^{(s_p+t)p^{e_p}})R_p$ for almost all $p$. Since $x_{1,p},\dots,x_{i,p}$ is a $\operatorname{p}$-standard sequence of $R_p$ for almost all $p$ by Proposition \ref{proposition of p-standard s.o.p. -> p-standard seq} and Proposition \ref{proposition approximations of p-standard s.o.p.}, we have $(x'_p)^{p^{e_p}}y_p^{p^{e_p}}\in (x_{1,p}^{(t+1)p^{e_p}},\dots,x_{i,p}^{(t+1)p^{e_p}})R_p$ for almost all $p$ by \cite[Proposition 2.4]{Kaw}. Therefore we have $x'_py_p\in (x_{1,p}^{t+1},\dots,x_{i,p}^{t+1})F^{e_p}_*R_p$ for almost all $p$. Hence, we have $x'y\in (x_1^{t+1},\dots,x_i^{t+1})F^{\epsilon}_*R_\infty$ and $\eta=0$ in $\H_{(x_1,\dots,x_i)}^i(F^\epsilon_*R_\infty)$. Then the conclusion follows from the following commutative diagram:
	\[
		\xymatrix{
			\H_\m^i(F^\epsilon_*R_\infty) \ar@{^{(}->}[r] \ar[d] & \H_{(x_1,\dots,x_i)}^i(F^\epsilon_*R_\infty) \ar@{^{(}->}[d] \\
			\ulim_p \H_{\m_p}^i(F^{e_p}_*R_p) \ar@{^{(}->}[r] & \ulim_p \H_{(x_{1,p},\dots,x_{i,p})}^i(F^{e_p}_*R_p)
		}.
	\]
\end{proof}
\section{Ultra-$F$-purity}
We introduce a new notion, ultra-$F$-pure singularities, and show that dense $F$-pure type descends under pure ring extensions.
\begin{setting} \label{setting of pairs}
	Let $(R,\ba^t)$ be a pair consisted of the following data:
	\begin{enumerate}
		\item $(R,\m)$ a reduced local ring essentially of finite type over $\C$ of dimension $d$,
		\item $\ba\subseteq R$ an ideal such that $\ba\cap R^\circ\neq \emptyset$,
		\item $t>0$ a real number.
	\end{enumerate}
\end{setting}
\begin{defn}
	With notation as in Setting \ref{setting of pairs}, $(R,\ba^t)$ is said to be {\it sharply ultra-$F$-pure} if for all $\epsilon_0 \in {^*\N}$, there exist $\epsilon\ge \epsilon_0$ and $f\in \ba^{\lceil(t\pi^{\epsilon}-1)\rceil}$ such that $fF^\epsilon:R\to R_\infty$ is pure. We simply say that $R$ is {\it ultra-$F$-pure} if $(R,R^t)$ is sharply ultra-$F$-pure.
\end{defn}
\begin{rem}
	\begin{enumerate}
		\item This definition depends on a choice of ultrafilter on $\mathcal{P}$ and isomorphism $\ulim_{p}\bar{\F_p}\cong \C$.
		\item $R$ is ultra-$F$-pure if and only if $R\to R^{\upf}$ is pure by Proposition \ref{proposition ultra-perfect closure limit}.
	\end{enumerate}
\end{rem}
\begin{eg}
	Let $R=(\C[x,y,z]/(x^3+y^3+z^3))_{(x,y,z)}$. If $R_p=(\bar{\F_p}[x,y,z]/(x^3+y^3+z^3))_{(x,y,z)}$, then $(R_p)_p$ is an approximation of $R$ for any non-principal ultrafilter $\mathcal{F}$ on $\mathcal{P}$ and any isomorphism $\ulim_p \bar{\F_p}\cong \C$. We observe that $R_p$ is $F$-pure if and only if $p \equiv 1 \pmod 3$. Then $R$ is ultra-$F$-pure if and only if $\{p\in \mathcal{P}|p\equiv 1 \pmod 3\}\in \mathcal{F}$ (cf. Proposition \ref{propositon dense F-pure type -> ultra-F-pure} and the proof of Proposition \ref{proposition quasi-Gorenstein + ultra-F-pure -> dense F-pure type}).
\end{eg}
\begin{lem}\label{lemma of field isomorpisms}
	Let $F$ be a subfield of $\C$ such that $F/\Q$ is finitely generated field extension. Given two field homomorphisms $f, g:F\to \C$. Then there exists an field automorphism $\alpha$ of $\C$ such that $g=\alpha \circ f$.
\end{lem}
\begin{proof}
	Let $\{e_i\}_{i=1}^n$ be a transcendental basis of $F/\Q$. Take $\{a_\lambda\}, \{b_\lambda\} \subseteq \C$ such that $\{f(e_i)\}\cup \{a_\lambda\}$ and $\{g(e_i)\}\cup \{b_\lambda\}$ are transcendental bases of $\C/\Q$. Let $\beta:\C\to \C$ be an automorphism of $\C$ such that $\beta(f(e_i))=g(e_i)$ and $\beta(a_\lambda)=b_\lambda$ for any $i$ and $\lambda$. We define $G$ to be $\Q(g(e_1),\dots,g(e_n))$. Then $\beta(f(F))/G$ and $g(F)/G$ are finite extensions and $g\circ(\beta\circ f)^{-1}:\beta(f(F))\to g(F)$ is an isomorphism which fixes $G$. Hence, there exists an automorphism $\gamma$ of $\C$ such that $\gamma$ is an extension of $g\circ (\beta\circ f)^{-1}$. Then $\alpha:=\gamma\circ \beta$ is a desired automorphism.
\end{proof}
\begin{prop}\label{propositon dense F-pure type -> ultra-F-pure}
	With notation as in Setting \ref{setting of pairs}, if $(R,\ba^t)$ is of dense sharply $F$-pure type, then there exist a non-principal ultrafilter $\mathcal{F}$ on $\mathcal{P}$ and an isomorphism $\alpha:\ulim_p \bar{\F_p}\cong \C$ such that $(R_p,\ba_p^t)$ is sharply $F$-pure for almost all $p$.
\end{prop}
\begin{proof}
	Suppose that $R$ is a localization of a finitely generated $\C$-algebra $S$ at a prime ideal $\p$, and $\bb=\ba\cap S$. Since $(R,\ba^t)$ is of dense sharply $F$-pure type, there exists a model $(A,S_A,\p_A,\bb_A)$ such that there exists a subset $D$ of $\operatorname{Spm} A$ such that $D$ is dense in $\Spec A$ and for all $\mu\in \Spec A$, a pair $(R_\mu,\ba_\mu^t)$ is sharply $F$-pure. We define $\phi:\operatorname{Spm} A\to \N$ by $\phi(\mu):=\operatorname{char} A/\mu$ and we can write $A\setminus \{0\}=\{x_i\}_{i=1}^\infty$ since $A$ is countable. Inductively take a sequence $\{\mu_i\}_{i=1}^{\infty}\subseteq D$ such that $\phi(\mu_i)>\phi(\mu_{i-1})$ and $x_1,\dots, x_{i-1}\notin\mu_i$. Let $p_i:=\phi(\mu_i)$ and take a non-principal ultrafilter $\mathcal{F}$ on $\mathcal{P}$ such that $\{p_i|i\in \N\}\in \mathcal{F}$ (see Proposition \ref{existence of ultrafilter}). For any  $i\in\N$, we define $\gamma_{p_i}:A\to \bar{\F_{p_i}}$ to be the composite morphism $A\to A/{\mu_i} \to \bar{\F_{p_i}}$. Note that $\gamma_p$ is defined for almost all $p$ in this setting. Then we have a ring homomorphism $\ulim_p \gamma_p:A\to \ulim_p \bar{\F_p}$. Note that $\ulim_p \gamma_p$ is injective since for any $x\in A\setminus\{0\}$, there exists $i_0\in \N$ such that for any $i\ge i_0$, $x\notin \mu_{i}$ by construction. Since $\ulim_p \bar{\F_p}\cong \C$, there exists an isomorphism $\alpha$ such that the diagram
	\[
		\xymatrix{
			 &A \ar[ld]_-{\ulim_p\gamma_p} \ar@{^{(}->}[rd]&\\
			  \ulim_p \bar{\F_p}\ar[rr]^-{\alpha}& &\C
	  	}
	\]
	commutes by Lemma \ref{lemma of field isomorpisms}. Then $(\gamma_p)_{p}$ defined above coincides with one in \cite[Lemma 4.9]{Sch04}. Hence, approximation $S_p$ of $S$ with respect to $\mathcal{F}$ and $\alpha$ is isomorphic to $S_{\mu}\otimes \bar{\F_{p}}$ for almost all $p$. By Proposition \ref{proposition of non-f-pure ideal} and the proof of \cite[Proposition 5.5]{Yam}, $(R_{p_i},\ba_{p_i}^t)$ is sharply $F$-pure for all $i$. Hence, $(R_p,\ba_p^t)$ is sharply $F$-pure for almost all $p$.
\end{proof}
\begin{prop}\label{proposition F-pure for almost all p -> ultra-F-pure}
	With notation as in Setting \ref{setting of pairs}, suppose that $\ba=(f_1,\dots,f_n)$ and $(R_p,\ba_p^t)$ is sharply $F$-pure for almost all $p$. Then for any $\epsilon_0\in {^*\N}$, there exist $\epsilon \ge \epsilon_0$ and $\mu_1,\dots,\mu_n\in {^*\N}$ such that $\mu_1+\dots+\mu_n=\lceil t(\pi^\epsilon-1)\rceil$ and $fF^\epsilon:R\to R_\infty$ is pure, where $f=\prod_{i=1}^n f_i^{\mu_i}$. In particular, $(R,\ba^t)$ is sharply ultra-$F$-pure.
\end{prop}
\begin{proof}
	Take any $\epsilon_0=\ulim_p e_{0,p}\in {^*\N}$. Since $(R_p,\ba_p^t)$ is sharply $F$-pure for almost all $p$, we have
	\[
		\sum_{e\ge e_{0,p}} \sum_{\phi} \phi(F^e_*\ba_p^{\lceil t(p^e-1)\rceil})=R_p,
	\]
	for almost all $p$, where $\phi$ runs through all elements of $\Hom_{R_p}(F^e_*R_p,R_p)$. Hence, for almost all $p$, there exist $e_p\in \N$ and $\phi_p\in \Hom_{R_p}(F^{e_p}_*R_p,R_p)$ such that $\phi_p(F^{e_p}_*\ba_p^{\lceil t(p^{e_p}-1)\rceil})=R_p$. Since
	\[
		\phi_p(F^{e_p}_*\ba_p^{\lceil t(p^{e_p}-1) \rceil})=\sum_{\substack{m_1,\dots,m_n\\m_1+\dots+m_n=\lceil t(p^{e_p}-1)\rceil}}\phi_p(F^{e_p}_*f_{1,p}^{m_1}\dots f_{n,p}^{m_n}R_p),
	\]
	there exist $m_{1,p}, \dots, m_{n,p}\in \N$ such that $m_{1,p}+\dots+m_{n,p}=\lceil t(p^{e_p}-1)\rceil$ and
	\[
		\phi_p(F^{e_p}_*f_{1,p}^{m_{1,p}}\dots f_{n,p}^{m_{n,p}}R_p)=R_p.
	\]
	Let $f_p:=f_{1,p}^{m_{1,p}}\dots f_{n,p}^{m_{n,p}}$, $f:=\ulim_p f_p$, $\epsilon:=\ulim_p e_p$ and $\mu_i:=\ulim_p m_{i,p}$. Then we have $f_pF^{e_p}:R_p\to R_p$ is pure for almost all $p$. It is enough to show the cyclic purity by \cite{Hoc}. Take any ideal $I\subseteq R$ and $x\in R$ such that $fF^{\epsilon}(x)\in fF^\epsilon(I)R_\infty$. Then $f_pF^{e_p}(x_p)\in f_pI_p^{[p^{e_p}]}$ for almost all $p$. Since $f_pF^{e_p}:R_p\to R_p$ is pure for almost all $p$, $x_p\in I_p$ for almost all $p$. Then we have $x\in I=IR_\infty\cap R$ since $R\to R_\infty$ is faithfully flat, which completes the proof.
\end{proof}
\begin{defn}
	Let $R$ be a local normal $\Q$-Gorenstein domain and $r$ be the minimum positive integer such that $rK_R=\Div(f)$ is Cartier, and fix a canonical ideal $R(K_R)=\omega_R\subseteq R$. We define a canonical covering $\tilde{R}$ of $R$ to be a $\Z/r\Z$-graded $R$-algebra
	\[
		\bigoplus_{i=0}^{r-1} \omega_R^{(i)}t^i,
	\]
	where $\omega_R^{(i)}$ is the $i$-th symbolic power of $\omega_R$ and $t^r=1/f$.
\end{defn}
\begin{rem}
	A canonical covering $\tilde{R}$ of $R$ is a local normal quasi-Gorenstein domain.
\end{rem}
\begin{lem}\label{Lemma of the non-standard hull of canonical coverings}
	Let $R$ be a $\Q$-Gorenstein normal local domain essentially of finite type over $\C$. Let $r$ be the minimum positive integer such that $rK_R$ is Cartier and let $\tilde{R}$ be a canonical cover of $R$. Then $(\tilde{R})_\infty\cong R_\infty\otimes_R \tilde{R}$.
\end{lem}
\begin{proof}
	Let $S$ be a normal $\Q$-Gorenstein domain of finite type over $\C$ such that $\omega_S^{(r)}$ is free and $\p\in\Spec S$ such that $R\cong S_\p$. Let $\tilde{S}\cong \bigoplus_{i=0}^{r-1}\omega_S^{(i)}t^{i}$ such that $(\tilde{S})_\p\cong \tilde{R}$. Take a $\Z$-subalgebra $A$ of $\C$ such that there exist models $(A,S_A)$ and $(A,\tilde{S}_A)$ as in \cite[Theorem 3.8]{Yam22a}. Since $(\tilde{R})_p\cong ((\tilde{S})_p)_{\p_p}$ for almost all $p$, it is enough to show $(\tilde{R})_\infty \cong R_\infty \otimes_S \tilde{S} $. This follows from the fact that the reductions modulo $p \gg 0$ of $\tilde{S}$ as a finite $S$-module coincide with those as a ring of finite type over $\C$.
\end{proof}
\begin{prop}\label{proposition canonical covering ultra-F-pure}
	With notation as in Setting \ref{setting of pairs}, suppose that $R$ is $\Q$-Gorenstein normal, $(R,\ba^t)$ is sharply ultra-$F$-pure, and $\tilde{R}$ is a canonical covering of $R$. Then $(\tilde{R},(\ba \tilde{R})^t)$ is sharply ultra-$F$-pure.
\end{prop}
\begin{proof}
	For any $\epsilon_0\in {^*\N}$, there exist $\epsilon=\ulim_p e_p \in {^*\N}$ and $f\in \ba^{\lceil t(\pi^\epsilon-1) \rceil}$ such that $fF^\epsilon: R\to R_\infty$ is pure.
	Let $r$ be the minimum positive integer such that $rK_R$ is Cartier. Let $x\in \tilde{S}$ be a homogeneous element with $\deg x=i$ and $\epsilon \in {^*\N}$. Then $F^\epsilon(x)$ is a homogeneous element of degree $i \pi^{\epsilon}\mod r$, where $j\equiv i \pi^{\epsilon} \mod r$ if and only if $j\equiv ip^{e_p}\mod r$ for almost all $p$. Since $p$ does not divide $r$ for almost all $p$, if $i\not\equiv 0 \mod r$, then we have $i\pi^{\epsilon} \not\equiv 0 \mod r$. Hence, we have the commutative diagram
	\[
	\xymatrix{
		\tilde{R} \ar[r]^-{\cdot F^\epsilon_*f} \ar[d]_{\operatorname{pr}_0}& F^\epsilon_* (\tilde{R})_\infty \ar[d]^{F^\epsilon_*\operatorname{pr}_0} \\
		R \ar[r]^{\cdot F^\epsilon_* f} & F^\epsilon_*R_\infty
	},
	\]
	where $\operatorname{pr}_0$ are the 0-th projections with respect to $\Z/r\Z$-grading induced by the definition of $\tilde{R}$ and Lemma \ref{Lemma of the non-standard hull of canonical coverings}. Tensoring the above diagram with $H_{\m}^d(\omega_R)$, we have
	\[
	 	\xymatrix@C=60pt{
	 	\H_{\m}^d(\omega_R)\otimes_R\tilde{R} \ar[r]^-{\id\otimes(\cdot F^{\epsilon}_*f)} \ar[d] & \H_\m^d(\omega_R)\otimes_R F^\epsilon_* (\tilde{R}_\infty) \ar[d]\\
	 	\H_\m^d(\omega_R) \ar[r]^-{\id \otimes (\cdot F^\epsilon_*f)} & \H_\m^d(\omega_R)\otimes_R F^\epsilon_*R _\infty	
 	}.
	\]
	Note that $\H_\m^d(\omega_R)\otimes_R \tilde{R}\cong \H_\m^d(\omega_{\tilde{R}})$. Take $\eta\in \H_\m^d(\omega_{R})\otimes_R \tilde{R}$ such that $(\id\otimes(\cdot F^\epsilon_*f))(\eta)=0$. Since $\operatorname{Soc}_R \H_\m^d(\omega_R)=\operatorname{Soc}_{\tilde{R}} \H_\m^d(\omega_{\tilde{R}})$ by \cite[Lemma 2.3]{Hara}, we may assume that $\eta\in \operatorname{Soc}_R \H_\m^d(\omega_R)$. Since $fF^\epsilon: R\to R_\infty$ is pure, the bottom horizontal morphism is injective and we have $\eta=0$. Hence, the top horizontal morphism is also injective.
\end{proof}
\begin{prop}\label{propostion pure subring of ultra-F-pure is ultra-F-pure}
	With notation as in Setting \ref{setting of pairs}, suppose that $R\to S$ is a pure local $\C$-algebra homomorphism between reduced local rings essentially of finite type over $\C$. If $(S_p,(\ba_p S_p)^t)$ is sharply $F$-pure for almost all $p$, then $(R,\ba^t)$ is sharply ultra-$F$-pure.
\end{prop}
\begin{proof}
	By Proposition \ref{proposition F-pure for almost all p -> ultra-F-pure}, for any $\epsilon_0\in {^*\N}$, there exist $\epsilon\ge \epsilon_0$ and $f\in \ba^{\lceil t(\pi^\epsilon-1) \rceil}$ such that $fF^\epsilon:S\to S_\infty$ is pure. Then we have a commutative diagram
	\[
		\xymatrix{
			R \ar[r]^-{fF^\epsilon} \ar[d] & R_\infty \ar[d] \\
			S \ar[r]^-{fF^\epsilon} & S_\infty
	}.
	\]
	Since $R\to S$ and $fF^\epsilon:S\to S_\infty$ are pure, $fF^\epsilon:R\to R_\infty$ is also pure. Therefore, $(R,\ba^t)$ is sharply ultra-$F$-pure.
\end{proof}
\begin{prop}\label{proposition quasi-Gorenstein + ultra-F-pure -> dense F-pure type}
	With notation as in Setting \ref{setting of pairs}, suppose that $R$ is quasi-Gorenstein and $(R,\ba^t)$ is sharply ultra-$F$-pure. Then $(R,\ba^t)$ is of dense sharply $F$-pure type.
\end{prop}
\begin{proof}
	Suppose that $(R_p,\ba_p^t)$ is not sharply $F$-pure for almost all $p$. Then, for almost all $p$, there exists $e_{0,p}\in \N$ such that for any $e\ge e_{0,p}$ and any $f\in \ba_p^{\lceil t(p^e-1)\rceil}$, $fF^e:R\to R$ is not pure. Let $\epsilon_0:=\ulim_p e_{0,p}$. Since $(R,\ba^t)$ is sharply ultra-$F$-pure, there exist $\epsilon=\ulim_p e_p\ge \epsilon_{0}$ and $f\in \ba^{\lceil t(\pi^\epsilon-1) \rceil}$ such that $fF^\epsilon:R\to R_\infty$ is pure. Then we have a commutative diagram
	\[
		\xymatrix@C=60pt{
		\H_\m^d(R) \ar@{^{(}->}[r]^-{\cdot F^{\epsilon}_*f} \ar[d] & \H_\m^d(F^\epsilon_* R_\infty) \ar@{^{(}->}[d]\\
		\ulim_p \H_{\m_p}^d(R_p) \ar[r]^-{\cdot(\ulim_p F^{e_p}_*f_p) } & \ulim_p \H_{\m_p}^d(F^{e_p}_*R_P)
	},
	\]
	where the injectivity of the right vertical morphism follows from Proposition \ref{proposition local cohomology ultraproducts injective}.  Let $\eta$ be a nonzero element of $\operatorname{Soc}_R \H_\m^d(R)$. Then $\eta_p\in \operatorname{Soc}_{R_p}\H_{\m_p}^d(R_p)$ for almost all $p$. Since $f_pF^{e_p}:R_p \to R_p$ is not pure for almost all $p$, the image of $\ulim_p \eta_p$ in $\ulim_p \H_{\m_p}^d(F^{e_p}_*R_p)$ is zero. This is a contradiction. Hence, $(R_p,\ba_p^t)$ is sharply $F$-pure for almost all $p$. Comparing approximations with reductions modulo $p$, $(R,\ba^t)$ is of dense sharply $F$-pure type (cf. Proposition \ref{proposition of non-f-pure ideal} and the proof of Proposition \ref{propositon dense F-pure type -> ultra-F-pure}).
\end{proof}
\begin{thm}
	With notation as in Setting \ref{setting of pairs}, suppose that $R$ is $\Q$-Gorenstein normal, $S$ is a reduced local ring essentially of finite type $\C$, and $R\to S$ is a pure local $\C$-algebra homomorphism. If $(S,(\ba S)^t)$ is of dense sharply $F$-pure type, then $(R,\ba^t)$ is also of dense sharply $F$-pure type.
\end{thm}
\begin{proof}
	If $(S,(\ba S)^t)$ is of dense sharply $F$-pure type, then there exist a non-principal ultrafilter $\mathcal{F}$ on $\mathcal{P}$ and an isomorphism $\alpha:\ulim_p \bar{\F_{p}} \cong \C$ such that $(S_p,(\ba_p S_p)^t)$ is sharply $F$-pure for almost all $p$ by Proposition \ref{propositon dense F-pure type -> ultra-F-pure}. Since $R\to S$ is a pure local $\C$-algebra homomorphism, $(R,\ba^t)$ is ultra-$F$-pure by Proposition \ref{propostion pure subring of ultra-F-pure is ultra-F-pure}. Let $\tilde{R}$ be a canonical covering of $R$. Then $(\tilde{R},(\ba \tilde{R})^t)$ is sharply ultra-$F$-pure by Proposition \ref{proposition canonical covering ultra-F-pure}. Since $\tilde{R}$ is quasi-Gorenstein and $(\tilde{R},(\ba \tilde{R})^t)$ is sharply ultra-$F$-pure, $(\tilde{R},(\ba \tilde{R})^t)$ is of dense sharply $F$-pure type by Proposition \ref{proposition quasi-Gorenstein + ultra-F-pure -> dense F-pure type}. Since $R\to \tilde{R}$ is finite and split, $(R,\ba^t)$ is also of dense sharply $F$-pure type.
\end{proof}

\section{Ultra-$F$-injectivity}
In this section, we discuss a non-standard variant of $F$-injectivity in the same setting as Setting \ref{setting of pairs}.
\begin{defn}
	With notation as in Setting \ref{setting of pairs}, $(R,\ba^t)$ is said to be {\it sharply ultra-$F$-injective} if for any integer $i$, for any nonzero element $\eta \in \H_{\m}^i(R)$ and for any $\epsilon_0\in {^*\N}$, there exist $\epsilon\ge \epsilon_0$ and $f\in \ba^{\lceil t(\pi^\epsilon -1) \rceil}$ such that the image of $\eta$ under the following composite morphism
	\[
		\H_\m^i(R)\to \H_\m^i(F^\epsilon_*R_\infty)\xrightarrow{\cdot F^\epsilon_*f} \H_\m^i(F^\epsilon_*R_\infty)
	\]
	is nonzero.
\end{defn}
\begin{rem}
	If $\ba=R$, then $(R,\ba^t)$ is sharply ultra-$F$-injective if and only if
	\[
		\H_\m^i(R) \to \H_\m^i(R^{\upf})
	\]
	is injective for all $i\in \Z$. When this condition holds, we say that $R$ is ultra-$F$-injective.
\end{rem}
\begin{prop}\label{proposition dense F-inj type -> F-inj for a.a.p}
	With notation as in Setting \ref{setting of pairs}, if $(R,\ba^t)$ is of dense sharply $F$-injective type, then there exist a non-principal ultrafilter $\mathcal{F}$ on $\mathcal{P}$ and an isomorphism $\alpha:\ulim_p \bar{\F_p}\cong \C$ such that $(R_p,\ba_p^t)$ is sharply $F$-injective for almost all $p$.
\end{prop}
\begin{proof}
	  By Proposition \ref{proposition of non-F-injective ideal}, the conclusion follows from an argument similar to Proposition \ref{propositon dense F-pure type -> ultra-F-pure}.
\end{proof}
\begin{prop}\label{proposition F-inj for a.a.p -> ultra-F-inj}
	With notation as in Setting \ref{setting of pairs}, assume that $\ba=(f_1,\dots,f_n)$ and $(R_p,(\ba_p)^t)$ is sharply $F$-injective for almost all $p$. Then for any $i\in \Z$, for any nonzero element $\eta \in \H_\m^i(R)$ and for any $\epsilon_0\in {^*\N}$, there exist $\epsilon\ge \epsilon_0$ and $\mu_1,\dots,\mu_n\in {^*\N}$ such that $\mu_1+\dots+\mu_n=\lceil t(\pi^\epsilon-1)\rceil$ and the image of $\eta$ under the following composite morphism
	\[
	\H_\m^i(R)\to \H_\m^i(F^\epsilon_*R_\infty)\xrightarrow{\cdot F^\epsilon_*f} \H_\m^i(F^\epsilon_*R_\infty)
	\]
	is nonzero, where $f=\prod_{i=1}^n f_i^{\mu_i}$. In particular, $(R,\ba^t)$ is sharply ultra-$F$-injective.
\end{prop}
\begin{proof}
	Take any $i\in \Z$, any nonzero element $\eta\in \H_\m^i(R)$ and any $\epsilon_0=\ulim_p {e_{0,p}}\in {^*\N}$. Since $\H_\m^i(R)\to \ulim_p \H_{\m_p}^i(R_p)$ is injective by Proposition \ref{proposition local cohomology ultraproducts injective}, $\eta_p \in \H_{\m_p}^i(R_p)$ is nonzero for almost all $p$. By the assumption, for almost all $p$, there exist $e_p\ge e_{0,p}$ and $m_{1,p},\dots, m_{n,p}\in \N$ such that $m_{1,p}+\dots +m_{n,p}=\lceil t(p^{e_p}-1)\rceil$ and the image of $\eta_p$ under the morphism $\H_{\m_p}^i(R_p) \xrightarrow{\cdot F^{e_p}_*f_p} \H_{\m_p}^i(F^{e_p}_*R_p) $ is nonzero, where $f_p:=f_{1,p}^{m_{1,p}}\cdots f_{n,p}^{m_{n,p}}$. Let $\epsilon=\ulim_p e_p$ and $f=\ulim_p f_p$. Then the image of $\eta$ under the morphism 
	\[
		\H_\m^i(R) \xrightarrow{\cdot F^{\epsilon}_*f} \H_\m^i(F^\epsilon_*R_\infty)
	\]
	is nonzero, which completes the proof.
\end{proof}
\begin{defn}[{\cite[Section 2]{CGM16}}]
	Let $R$ be a ring and $S$ be an $R$-algebra. A ring homomorphism $R\to S$ is said to be {\it strongly pure} if for any $\q\in \Spec S$, $R_{\q\cap R}\to S_\q$ is pure.
\end{defn}
\begin{rem}
	If $R\to S$ is faithfully flat, then $R\to S$ is strongly pure.
\end{rem}
Since strongly pure morphisms are a somewhat limited class, we consider the following condition enough strong to show the descent of ultra-$F$-injectivity.
\begin{defn}
	Let $(R,\m)$ be a local ring and $S$ be an $R$-algebra. A ring homomorphism $R\to S$ is said to satisfy the condition (*) if there exists a prime ideal $\q$ of $S$ minimal among primes of $S$ lying over $\m$ such that $R\to S_\q$ is pure.
\end{defn}
\begin{eg}
	\begin{enumerate}
		\item Let $R=(\C[xy,xz])_{(xy,xz)}$, $S=(\C[x,y,z])_{(x,y,z)}$, $\q_1=(x)S$ and $\q_2=(y,z)$. Then $\q_1\cap R=\q_2\cap R=(xy,xz)R$. $R \to S_{\q_1}$ is not pure and $R\to S_{\q_2}$ is pure. Hence, $R\to S$ satisfies the condition (*) but is not strongly pure.
		\item Let $R=(\C[xz,xw,yz,yw])_{(xz,xw,yz,yw)}$, $S=(\C[x,y,z,w])_{(x,y,z,w)}$. Let $\q$ be a prime ideal of $S$ minimal among primes of $S$ lying over $\m$. Then we have $\q=(x,y)R$ or $\q=(z,w)R$. In both cases, $R\to S_\q$ is not pure. Hence, $R\to S$ is pure but does not satisfy the condition (*).
	\end{enumerate}
\end{eg}
\begin{prop}\label{proposition ultra-F-inj pure subring}
	With notation as in Setting \ref{setting of pairs}, suppose that $S$ is a reduced local ring essentially of finite type over $\C$, and a local $\C$-algebra homomorphism $R\to S$ satisfies the condition (*). If $(S_p,(\ba_p S_p)^t)$ is sharply $F$-injective for almost all $p$, then $(R,\ba^t)$ is sharply ultra-$F$-injective.
\end{prop}
\begin{proof}
	We argue similarly to \cite[Theorem 3.8]{DM22}. Let $\q$ be a prime ideal that is minimal among primes of $S$ lying over $\m$ and such that $R\to S_\q$ is pure. Then $R\to S_\q$ is pure and $S_\q/\m S_\q$ is of dimension zero. 
	Take any $i\in \Z$, any nonzero element $\eta\in \H_\m^i(R)$ and any $\epsilon_0\in {^*\N}$. Since $R\to S_\q$ is pure, the image of $\eta$ under the morphism $\H_\m^i(R)\to H_{\m S_\q}^i(S_\q)\cong \H_{\q S_\q}^i(S_\q)$ is nonzero. By Proposition \ref{proposition F-inj for a.a.p -> ultra-F-inj}, there exists $\epsilon \ge \epsilon_0$ and $f\in \ba^{\lceil t(\pi^\epsilon -1)\rceil}$ such that the image of $\eta$ under the composite morphism 
	\[
		\H_\m^i(R) \to \H_{\q S_\q}^i(S_\q) \xrightarrow{\cdot F^{\epsilon}_*f} \H_{\q S_\q}^i(F^\epsilon_*(S_\q)_\infty)
	\]
	is nonzero.
	Considering a commutative diagram
	\[
		\xymatrix{
			\H_\m^i(R) \ar[r] \ar[d]^-{\cdot F^\epsilon_* f} & \H_{\m S_\q}^i(S_\q)  \ar[r]^-{\cong} \ar[d]^-{\cdot F^\epsilon_* f}& \H_{\q S_\q}^i(S_\q) \ar[d]^-{\cdot F^\epsilon_* f} \\
			 \H_{\m}^i(F^\epsilon_* R_\infty) \ar[r] & \H_{\m S_\q}^i(F^\epsilon_* (S_\q)_\infty) \ar[r]^-\cong & \H_{\q S_\q}^i(F^\epsilon_* (S_\q)_\infty)
		},
	\]
	the image of $\eta$ under the morphism $\H_\m^i(R)\xrightarrow{\cdot F^\epsilon_* f} \H_\m^i(F^\epsilon_* R_\infty)$ is nonzero. Therefore, $(R,\ba)^t$ is sharply ultra-$F$-injective.
\end{proof}
\begin{prop}\label{proposition ultra-F-inj->dense F-injt type}
	With notation as in Setting \ref{setting of pairs}, assume that $(R,\ba^t)$ is sharply ultra-$F$-injective and $R/\m\cong \C$. Then $(R,\ba^t)$ is dense sharply $F$-injective type.
\end{prop}
\begin{proof}
	Suppose that $(R_p,(\ba_p)^t)$ is not sharply $F$-injective for almost all $p$. Then there exists $i\in \Z$ such that for almost all $p$, there exist $e_p\in \N$ and $f_p\in \ba_p^{\lceil t(p^{e_p}-1)\rceil}$ such that $\H_{\m_p}^i(R_p)\xrightarrow{\cdot F^{e_p}_*f_p} \H_{\m_p}^i(F^{e_p}_*R_p)$ is not injective. Let $\epsilon=\ulim_p e_p\in {^*\N}$ and $f=\ulim_p f_p$. Then we have a commutative diagram
	\[
		\xymatrix{
			\H_\m^i(R) \ar[r] \ar@{^{(}->}[d] & \H_\m^i(F^{\epsilon}_*R_\infty) \ar@{^{(}->}[d] \\
			\ulim_p \H_{\m_p}^i(R_p) \ar[r] & \ulim_p \H_{\m_p}^i(F^{e_p}_*R_p)
	},
	\]
	where the vertical maps are injective by Proposition \ref{proposition local cohomology ultraproducts injective}.
	\begin{claim}
		$\operatorname{Soc}_R\H_\m^i(R)\cong \operatorname{Soc}_{R_\infty}\left(\ulim_p \H_{\m_p}^i(R_p)\right)$.
	\end{claim}
	\begin{claimproof}
		Take a regular local ring $(S,\n)$ essentially of finite type over $\C$ such that $R$ is a homomorphic image of $S$ and let $t=\dim S$. Then
		\[
			\H_\m^i(R)\cong \H_\n^i(R) \cong \Hom_S (\Ext_S^{t-i}(R,S),E_S),
		\]
		where $E_S$ is the injective hull of $R/\m \cong S/\n$ as an $S$-module. Hence, we have
		\begin{eqnarray*}
			\operatorname{Soc}_S \H_\m^i(R) & \cong & \operatorname{Soc}_S  (\Hom_S (\Ext_S^{t-i}(R,S),E_S))\\
			&\cong & \Hom_S (\Ext_S^{t-i}(R,S),S/\n).
		\end{eqnarray*}
	Therefore, \[
	l_R(\operatorname{Soc}_R\H_\m^i(R))=l_R(\Hom_S(\Ext_S^{t-i}(R,S),S/\n)).
	\]
	On the other hand, 
	\[
	l_R(\Hom_S(\Ext_S^{t-i}(R,S),S/\n))=l_{R_p}(\Hom_{S_p}(\Ext_{S_p}^{t-i}(R_p,S_p),S_p/\n_p)
	\]
	for almost all $p$. By a similar argument, we have 
	\[
		l_R(\operatorname{Soc}_R \H_\m^i(R))=l_{R_p}(\operatorname{Soc}_{R_p} \H_{\m_p}^i(R_p))
	\] for almost all $p$. Hence, $\operatorname{Soc}_{R_\infty}(\ulim_p \H_{\m_p}^i(R_p))\cong \ulim_p(\operatorname{Soc}_{R_p} \H_{\m_p}^i(R_p))$ is a finite $R_\infty$-module of length $l_R(\operatorname{Soc}_{R}\H_\m^i(R))$. Since $R/\m\cong \C$ by the assumption, we have $R_\infty/\m R_\infty\cong \C$. Therefore, $\operatorname{Soc}_{R_\infty} (\ulim_p \H_{\m_p}^i(R_p))$ is also a finite $R$-module of length $l_R(\operatorname{Soc}_R\H_\m^i(R))$. By Proposition \ref{proposition local cohomology ultraproducts injective}, $\H_\m^i(R)\to \ulim_p \H_{\m_p}^i(R_p)$ is injective. Hence, the morphism
	\[
		\operatorname{Soc}_R \H_\m^i(R) \to \operatorname{Soc}_{R_\infty}(\ulim_p \H_{\m_p}^i(R_p))
	\]
	is an isomorphism.
	\end{claimproof}
	Since $\H_\m^i(R) \xrightarrow{\cdot F^{\epsilon}_*f} \H_\m^i(F^{\epsilon}_*R_\infty)$ is injective by the assumption,
	\[
		\ulim_p \H_{\m_p}^i(R_p) \xrightarrow{\cdot F^{e_p}_*f_p}\ulim_p \H_{\m_p}^i(F^{e_p}_*R_p)
	\]
	 is injective by the above claim. However, this is a contradiction. Hence, $(R_p,(\ba_p)^t)$ is sharply $F$-injective for almost all $p$. Comparing approximations with reductions modulo $p>0$, $(R,\ba^t)$ is of dense sharply $F$-injective type.
\end{proof}
Combining above propositions, we get the following theorem.
\begin{thm}
	With notation as in Setting \ref{setting of pairs}, suppose that $S$ is a reduced local ring essentially of finite type over $\C$, a local $\C$-algebra homomorphism $R\to S$ satisfies the condition (*) and $R/\m\cong \C$. If $(S,(\ba S)^t)$ is of dense sharply $F$-injective type, then $(R,\ba^t)$ is of dense sharply $F$-injective type.
\end{thm}
\begin{proof}
	By Proposition \ref{proposition dense F-inj type -> F-inj for a.a.p}, there exist a non-principal ultrafilter $\mathcal{F}$ on $\mathcal{P}$ and an isomorphism $\ulim_p\bar{\F_p}\cong \C$ such that $(S_p,(\ba_p S_p)^t)$ is sharply $F$-injective for almost all $p$. By Proposition \ref{proposition ultra-F-inj pure subring}, $(R,\ba^t)$ is sharply ultra- $F$-pure. Since $R/\m\cong \C$, by Proposition \ref{proposition ultra-F-inj->dense F-injt type}, $(R,\ba^t)$ is of dense sharply $F$-injective type.
\end{proof}
\begin{rem}
	We expect that the conclusion holds even if we only suppose that $R\to S$ is pure because it follows from the main result of Godfrey and Murayama \cite{GM22} and the weak ordinarity conjecture (see \cite{BST}).
\end{rem}

\end{document}